\pdfoutput=1
\documentclass{amsart}
\usepackage{amsthm}
\usepackage{amsmath}
\usepackage{amscd}
\usepackage{amssymb}
\usepackage{graphicx}
\usepackage{hyperref}
\usepackage{tikz}

\theoremstyle{plain}
\newtheorem{theorem}{Theorem}
\newtheorem{proposition}{Proposition}

\newtheorem{fact}{Fact}

\theoremstyle{definition}
\newtheorem{definition}{Definition}

\theoremstyle{remark}
\newtheorem{remark}{Remark}

\newcommand{\rot}{\operatorname{rot}}

\newcommand{\atrb}{~
\,\begin{picture}(10,13)
\put(5,3){\circle{15}}
\put(0,-2){\vector(1,1){10}}
\put(10,-2){\vector(-1,1){10}}
\put(5,10.5){\vector(0,-1){15}}
\end{picture}~~~
}

\begin{document}
\title[Legendrian fronts]{New invariants of Legendrian knots}
\author{Noboru Ito}
\address{National Institute of Technology, Ibaraki College, 866 Nakane, Hitachinaka, Ibaraki, 312-8508, Japan}
\email{nito@gm.ibaraki-ct.ac.jp}
\author{Masashi Takamura}
\address{School of Social Informatics, Aoyama Gakuin University, 5-10-1, Fuchinobe, Chuo-ku, Sagamihara-shi, anagawa 252-5258, Japan
}
\email{takamura@si.aoyama.ac.jp}
\keywords{Legendran front; Legendrian knot; contact structure; chord diagram}
\thanks{MSC2020: 57R42, 57K33}
\date{December 4, 2021}

\begin{abstract}
We give new functions of Legendrian knots derived from Legendrian fronts.  These are integer-valued linear functions that are alike the Arnold basic invariant of plane curves.  Various  generalizations of the Arnold basic  invariant have been known.  
In this paper, we give another extension of Arnold's idea.  
\end{abstract}
\maketitle
\section{Introduction}
Arnold \cite{Arnold1994} introduced integer-valued functions $J^+$, $J^-$, and $St$ for \emph{plane curves}, each of which is the image of a generic immersion $S^1$ $\to$ $\mathbb{R}^2$,  where the self-intersections are transverse double points.  A plane curve is regarded as an object called a \emph{Legendrian front} that is an  image of a projection of a \emph{Legendrian knot}.  
Arnold showed that $J^+$ is a Legendrian knot invariant.  Nowadays, functions $J^+$, $J^-$, and $St$ are called \emph{Arnold $($basic$)$ invariants}.   
Arnold invariants are of much interest to researchers dealing with some aspects, and have been studied a lot: several versions of explicit formulae  of them (Polyak \cite{Polyak1998}, Shumakovitch \cite{Shumakovich1995}, and Viro \cite{Viro1996}), generalizations to fronts (Aicardi \cite{Aicardi1997}, and Arnold \cite{Arnold1994}, Polyak \cite{Polyak1998}), and higher-order cases: Arakawa-Ozawa \cite{ArakawaOzawa1999} for $St$, Goryunov \cite{Goryunov1997} for $J^+$, and Viro \cite{Viro1996} for $J^-$ with a setting that Arnold  invariants are of lower orders.  


Arnold showed his function $J^+$  
is a Legendrian knot invariant \cite{Arnold1994}.  In \cite{HayanoIto2015}, Hayano and one of the author NI gave Fact~\ref{fact0}.  In Fact~\ref{fact0}, a deformation of type strong RI\!I (RI\!I\!I,~resp.) is called \emph{iR2-move} (\emph{R3-move},~resp.).  From here, the terminologies of this paper obey \cite{HayanoIto2015}.  
\begin{fact}[Hayano-Ito \cite{HayanoIto2015}]\label{fact0}
Let $r$ be an integer and $j^+$ an even integer.  For a plane curve $C$, let $\rot (C)$ be the rotation number and $J^+ (C)$ the Arnold invariant.   
Then, for any pair  $(r, j^+)$, there exists an infinite family of plane curves $\{C_\lambda~|~\lambda \in \mathbb{N} \}$ satisfying the following conditions:
\begin{itemize}
\item $\rot (C_\lambda)=r$ and $J^+ (C_\lambda)$ $=$ $j^+$ for any $\lambda$.  
\item For any $\lambda$ and $\mu$, $C_\lambda$ and $C_{\mu}$ are not equivalent under iR2-move  and R3-move.   
\end{itemize}
\end{fact}
In order to prove Fact~\ref{fact0}, we may use plane curves as in Fig.~\ref{abc} with the nonnegative rotation number $r$ ($=a$).  This example is similar to \cite{HayanoIto2015}.  This sequence $C(a, b, c)$ is parametrized by a tuple $(a, b, c)$ of integers $a$, $b$, and $c$.  Here, if $C(a, b, c)$ has the number of parts of type $(a)$ ($(b)$, $(c)$,~resp.) is $a_0$ ($b_0$, $c_0$,~resp.), we say that $C(a_0, b_0, c_0)$.   
Note that $J^+ (C(a, b, c))$ $=$ $a-2b+2c$.   
Further, it is proved that for fixed nonnegative integers $a_0$, $b_0$ and $c_0$ such that $a_0 -2b_0 + 2c_0=j^+$ and the rotation number $r$ ($=$ $a_0$), any two plane curves in the family $\{ C(r, b_0+k, c_0+k) \}_{k \in \mathbb{Z}_{>0}}$ are not equivalent under iR2-move and R3-move (cf.~\cite{HayanoIto2015}).    
\begin{figure}[h!]
\includegraphics[width=6cm]{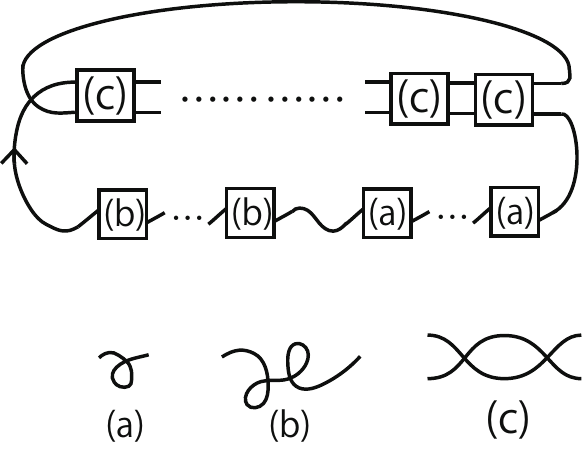}
\caption{$C(a, b, c)$ (upper) and parts (a), (b), and (c) (lower) with a  nonnegative rotation number.  The integer $a$ ($b$, $c$,~resp.) indicates the number of appearances of type $(a)$ ($(b)$, $(c)$,~resp.).  A plane curve $\dot{C}(a, b, c)$ is the corresponding $C(a, b, c)$ with the base point at the position on which the  arrow mark places.}\label{abc}
\end{figure}

Let $\dot{C}(a, b, c)$ be the plane curve $C(a, b, c)$ with the base point at the position on which the  arrow mark places as in Fig.~\ref{abc}.    Let $\dot{J}^{+}$ be the Arnold invariant of long curves, each of which is identified with a plane curve with a base point \cite[Definition~4.3]{Ito2010}.  
\begin{theorem}\label{main}
Let $r$ be a nonnegative integer, $j^+$ an integer, $\rot$ the rotation number, and $\dot{J}^+$ the Arnold invariant.  
There exist functions $I_{2, 3, k}$ $(1 \le k \le 5)$ that are invariant under iR2-move,  R3-move preserving the base point such that the function detects any two plane curves in the family $\{ \dot{C}(r, b_0+k, c_0+k) \}_{k \in \mathbb{Z}_{>0}}$ as in Fig.~\ref{abc}, where $\rot(\dot{C}(r, b_0+k, c_0+k))=r$ and $\dot{J}^+(\dot{C}(r, b_0+k, c_0+k))=j^+$.    
\end{theorem}
\begin{remark}
For long curves, the Arnold invariant $\dot{J}^+$ is formulated by Gusein-Zade \cite{Gusein-ZadeNatanzon1997} and  Zhou-Zou-Pan  \cite{ZhouZouPan1998}, independently.       
\end{remark}
\section{Preliminaries}\label{sec_pre} 
\subsection{Legendrian knots and fronts}
The reader who is familiar with \cite{HayanoIto2015} may skip this section;  here we pick some  definitions in \cite{HayanoIto2015}.        
\begin{definition}
A plane curve is the image of a generic immersion $S^1$ $\to$ $\mathbb{R}^2$, where the self-intersections are transverse double points.  
\end{definition} 
\begin{definition}[Legendrian knot $K_C$ associated with a plane curve $C$]
Let $C$ be a plane curve. 
For an given oriented $C$, we take a generic immersion $f : S^1$ $\to$ $\mathbb{R}^2$ so that $f(S^1)=C$ and the orientation of $C$ is induced by that of $S^1$.  Let $df$ be the derivative $TS^1 \to T \mathbb{R}^2$.  Since $f$ is an immersion, $f$ implies that $df(p) \neq 0$ for every $p$.  Since there exists the projection $\pi : T \mathbb{R}^2 \setminus \mathbb{R}^2$ $\to$ $UT \mathbb{R}^2$, and the $0$-section of $T \mathbb{R}^2$ is identified with $\mathbb{R}^2$, the map $\pi \circ df|_{S^1} :$ $S^1$ $\to$ $UT\mathbb{R}^2$ gives a knot in $UT \mathbb{R}^2$.  This knot is denoted by $K_C$ and  called the Legendrian knot associated with $C$.      
\end{definition}
\begin{fact}
If two plane curves $C_0$ and $C_1$ are equivalent under iR2-move and R3-move, then there exists an ambient isotopy in $S^3$ that deforms $K_{C_0}$ to $K_{C_1}$ and keeps the $(+2)$-framed unknot fixed.  In particular, $K_{C_0}$ to $K_{C_1}$ are isotopic as framed knots in $S^3$.    
\end{fact}
\section{Explicit relationship between plane curves and Legendrian knots} 
The realization of the Legendrian knot \cite{HayanoIto2015} gives one to one correspondence  between an iR2 (R3,~resp.)-move of plane curves and the Reidemeister move $\Omega_2$ ($\Omega_3$,~resp.) as in Fig.~\ref{ir2}  (Fig.~\ref{r3},~resp.) if the tangent vector of a branch of a self-tangency (triple point crossing,~resp.) is \emph{not} a horizontal direction vector oriented from right to left ($\ast$).  Note that the condition ($\ast$) can be excluded from the neighborhood of a digon (triangle,~resp.) by small isotopy for plane curves and for knots,  respectively. 
\begin{figure}[h!]
\includegraphics[width=6cm]{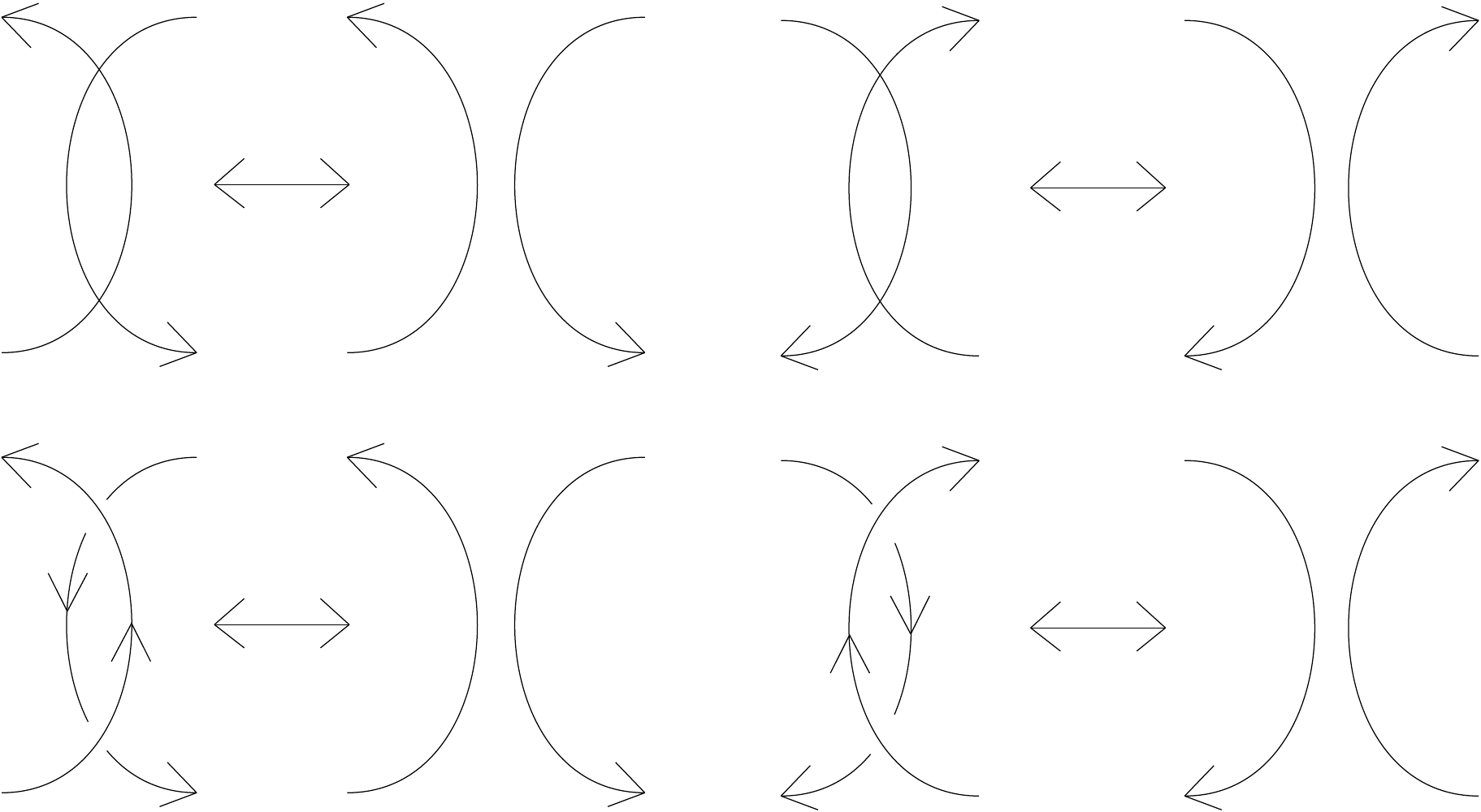}
\caption{iR2-moves}\label{ir2}
\end{figure}       
\begin{figure}[h!]
\includegraphics[width=6cm]{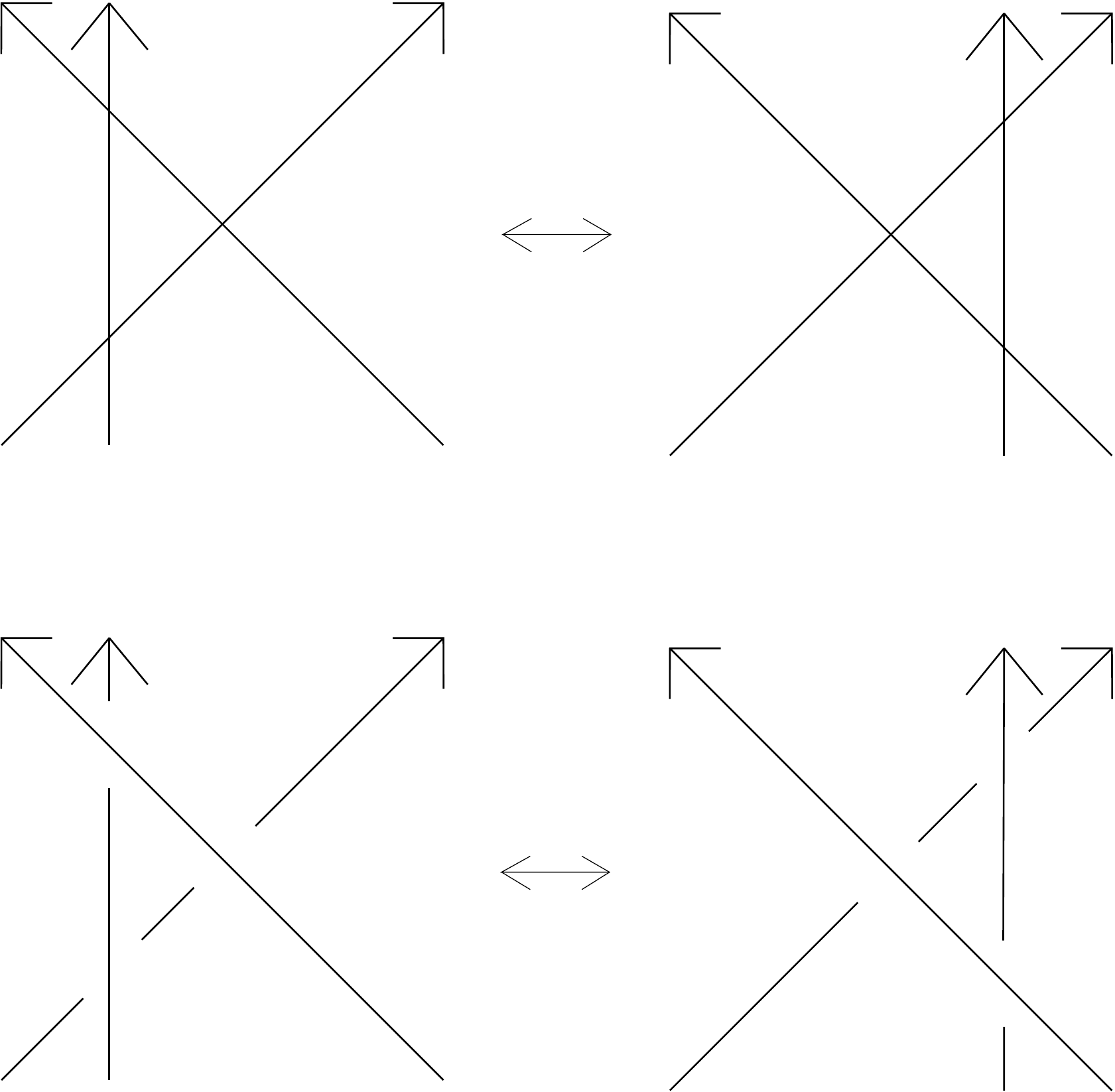}
\caption{An examples of R3-move}\label{r3}
\end{figure}
Given a plane curve $C$, we have a diagram $D_{K_C}$ of a knot $K_C$ keeping the $(+2)$-framed unknot fixed  \cite[Section~3.2 just after Prop.~3.2 to Prop.~3.4]{HayanoIto2015} (Fig.~\ref{LGtoAD}).   
\begin{figure}[h!]
\includegraphics[width=6cm]{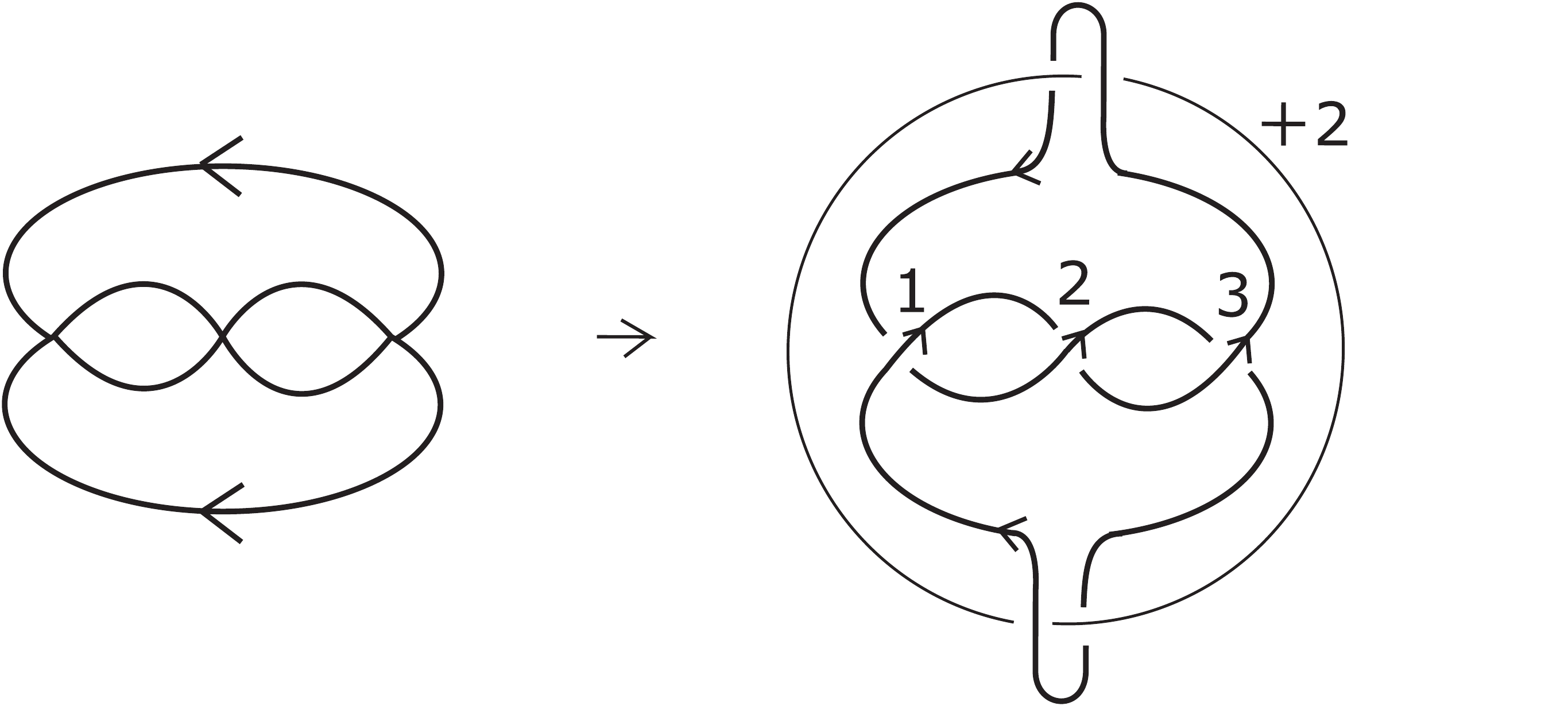}
\caption{An example $K_C$ (right) of  Legendrian knots derived from a plane curve $C$ (left)}\label{LGtoAD}
\end{figure}
\section{Proof of Theorem~\ref{main}}
We apply the same argument as \cite{ItoTakamura2021} with  relators \cite[Section~2.2, Definition~7]{ItoTakamura2021}.  
Although the meaning of ``the same argument" may be clear, we pose some comments.  

\noindent(Construction) 
First of all, if the reader does not know the definition of Gauss diagrams, see, e.g., \cite[Page~1046, Fig.~1]{GoussarovPolyakViro2000}). 
Second, though we have a Legendrian knot diagram $D_{K_C}$ from $K_C$, we use relators \cite[Section~2.2, Definition~7]{ItoTakamura2021} given by  positive knot diagrams derived from plane curves.  Third, we apply the same argument as \cite{ItoTakamura2021} to these relators.  Then we have invariants as in Theorem~\ref{main}.  

\noindent(Detection of Legendrian knots) 
Essentially, we count sub-arrow diagrams, each of which is isomorphic to $\atrb$.  
For every $(2, n)$-torus knot $T_n$ such that $n=2m+1$, the number of sub-arrow diagrams of type $\atrb$ is 
\[
\frac{1}{6} m (m+1) (2m+1).  
\]
Then the invariant detects two $T_n$ and $T_{n'}$ corresponding to $K_C$ and $K_{C'}$ respectively.  
\hfill$\Box$
\section{New functions of Legendrian fronts}
Since we give our formulas with simpler presentations, we slightly change the notation as in \cite{ItoTakamura2021}.  More precisely, we switch each arrow presentation to the signed chord in the way of Fig.~\ref{switch}.  
\begin{figure}[h!]
\includegraphics[width=6cm]{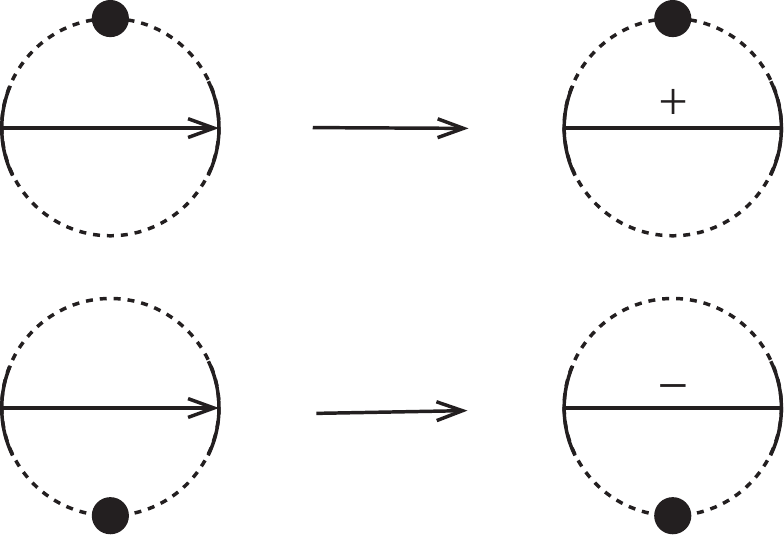}
\caption{Switching each arrow presentation to the signed chord.}\label{switch}
\end{figure}

In the following, we list six  invariants $I_{2, 1}$ and $I_{3, i}(C)$ ($1 \le i \le 5$) for a given plane curve $C$:    

\begin{align*}
I_{3, 1} &:= \begin{tikzpicture}[baseline=0pt]
\draw (0pt, 0pt) circle (7pt);
\draw (7pt, 0pt)--(3.5pt, 6.06218pt);
\draw (-3.5pt, 6.06218pt)--(-7pt, 0pt);
\draw (-3.5pt, -6.06218pt)--(3.5pt, -6.06218pt);
\fill (6.06218pt, -3.5pt) circle (1pt);
\end{tikzpicture}
+
\begin{tikzpicture}[baseline=0pt]
\draw (0pt, 0pt) circle (7pt);
\draw (7pt, 0pt)--(3.5pt, -6.06218pt);
\draw (3.5pt, 6.06218pt)--(-3.5pt, 6.06218pt);
\draw (-7pt, 0pt)--(-3.5pt, -6.06218pt);
\fill (6.06218pt, -3.5pt) circle (1pt);
\end{tikzpicture}
-
\begin{tikzpicture}[baseline=0pt]
\draw (0pt, 0pt) circle (7pt);
\draw (7pt, 0pt)--(3.5pt, 6.06218pt);
\draw (-3.5pt, 6.06218pt)--(-3.5pt, -6.06218pt);
\draw (-7pt, 0pt)--(3.5pt, -6.06218pt);
\fill (6.06218pt, -3.5pt) circle (1pt);
\end{tikzpicture}
-
\begin{tikzpicture}[baseline=0pt]
\draw (0pt, 0pt) circle (7pt);
\draw (7pt, 0pt)--(3.5pt, -6.06218pt);
\draw (3.5pt, 6.06218pt)--(-7pt, 0pt);
\draw (-3.5pt, 6.06218pt)--(-3.5pt, -6.06218pt);
\fill (6.06218pt, -3.5pt) circle (1pt);
\end{tikzpicture}
-
\begin{tikzpicture}[baseline=0pt]
\draw (0pt, 0pt) circle (7pt);
\draw (7pt, 0pt)--(-3.5pt, 6.06218pt);
\draw (3.5pt, 6.06218pt)--(-7pt, 0pt);
\draw (-3.5pt, -6.06218pt)--(3.5pt, -6.06218pt);
\fill (6.06218pt, -3.5pt) circle (1pt);
\end{tikzpicture}
+
\begin{tikzpicture}[baseline=0pt]
\draw (0pt, 0pt) circle (7pt);
\draw (7pt, 0pt)--(-3.5pt, -6.06218pt);
\draw (3.5pt, 6.06218pt)--(3.5pt, -6.06218pt);
\draw (-3.5pt, 6.06218pt)--(-7pt, 0pt);
\fill (6.06218pt, -3.5pt) circle (1pt);
\end{tikzpicture}
-
\begin{tikzpicture}[baseline=0pt]
\draw (0pt, 0pt) circle (7pt);
\draw (7pt, 0pt)--(-3.5pt, 6.06218pt);
\draw (3.5pt, 6.06218pt)--(-3.5pt, -6.06218pt);
\draw (-7pt, 0pt)--(3.5pt, -6.06218pt);
\fill (6.06218pt, -3.5pt) circle (1pt);
\end{tikzpicture}
+
\begin{tikzpicture}[baseline=0pt]
\draw (0pt, 0pt) circle (7pt);
\draw (7pt, 0pt)--(-7pt, 0pt);
\draw (3.5pt, 6.06218pt)--(-3.5pt, -6.06218pt);
\draw (-3.5pt, 6.06218pt)--(3.5pt, -6.06218pt);
\fill (6.06218pt, -3.5pt) circle (1pt);
\end{tikzpicture}
~, \\
I_{3, 2} &:= \begin{tikzpicture}[baseline=0pt]
\draw (0pt, 0pt) circle (7pt);
\draw (7pt, 0pt)--(-7pt, 0pt);
\draw (0pt, 7pt)--(0pt, -7pt);
\fill (4.94975pt, -4.94975pt) circle (1pt);
\draw[font=\tiny] (10.5pt, 0pt) node {\mbox{$-$}};
\draw[font=\tiny] (0pt, 10.5pt) node {\mbox{$+$}};
\end{tikzpicture}
-
\begin{tikzpicture}[baseline=0pt]
\draw (0pt, 0pt) circle (7pt);
\draw (7pt, 0pt)--(-7pt, 0pt);
\draw (0pt, 7pt)--(0pt, -7pt);
\fill (4.94975pt, -4.94975pt) circle (1pt);
\draw[font=\tiny] (10.5pt, 0pt) node {\mbox{$+$}};
\draw[font=\tiny] (0pt, 10.5pt) node {\mbox{$-$}};
\end{tikzpicture}
-
\begin{tikzpicture}[baseline=0pt]
\draw (0pt, 0pt) circle (7pt);
\draw (7pt, 0pt)--(-3.5pt, 6.06218pt);
\draw (3.5pt, 6.06218pt)--(3.5pt, -6.06218pt);
\draw (-7pt, 0pt)--(-3.5pt, -6.06218pt);
\fill (6.06218pt, -3.5pt) circle (1pt);
\end{tikzpicture}
+
\begin{tikzpicture}[baseline=0pt]
\draw (0pt, 0pt) circle (7pt);
\draw (7pt, 0pt)--(-3.5pt, -6.06218pt);
\draw (3.5pt, 6.06218pt)--(-3.5pt, 6.06218pt);
\draw (-7pt, 0pt)--(3.5pt, -6.06218pt);
\fill (6.06218pt, -3.5pt) circle (1pt);
\end{tikzpicture}
+
\begin{tikzpicture}[baseline=0pt]
\draw (0pt, 0pt) circle (7pt);
\draw (7pt, 0pt)--(-7pt, 0pt);
\draw (3.5pt, 6.06218pt)--(3.5pt, -6.06218pt);
\draw (-3.5pt, 6.06218pt)--(-3.5pt, -6.06218pt);
\fill (6.06218pt, -3.5pt) circle (1pt);
\end{tikzpicture}
-
\begin{tikzpicture}[baseline=0pt]
\draw (0pt, 0pt) circle (7pt);
\draw (7pt, 0pt)--(-3.5pt, -6.06218pt);
\draw (3.5pt, 6.06218pt)--(-7pt, 0pt);
\draw (-3.5pt, 6.06218pt)--(3.5pt, -6.06218pt);
\fill (6.06218pt, -3.5pt) circle (1pt);
\end{tikzpicture}
~,  \\
I_{3, 3} &:= \begin{tikzpicture}[baseline=0pt]
\draw (0pt, 0pt) circle (7pt);
\draw (7pt, 0pt)--(-7pt, 0pt);
\draw (0pt, 7pt)--(0pt, -7pt);
\fill (4.94975pt, -4.94975pt) circle (1pt);
\draw[font=\tiny] (10.5pt, 0pt) node {\mbox{$-$}};
\draw[font=\tiny] (0pt, 10.5pt) node {\mbox{$-$}};
\end{tikzpicture}
-
\begin{tikzpicture}[baseline=0pt]
\draw (0pt, 0pt) circle (7pt);
\draw (7pt, 0pt)--(-7pt, 0pt);
\draw (0pt, 7pt)--(0pt, -7pt);
\fill (4.94975pt, -4.94975pt) circle (1pt);
\draw[font=\tiny] (10.5pt, 0pt) node {\mbox{$+$}};
\draw[font=\tiny] (0pt, 10.5pt) node {\mbox{$+$}};
\end{tikzpicture}
+
\begin{tikzpicture}[baseline=0pt]
\draw (0pt, 0pt) circle (7pt);
\draw (7pt, 0pt)--(-3.5pt, 6.06218pt);
\draw (3.5pt, 6.06218pt)--(-3.5pt, -6.06218pt);
\draw (-7pt, 0pt)--(3.5pt, -6.06218pt);
\fill (6.06218pt, -3.5pt) circle (1pt);
\end{tikzpicture}
-
\begin{tikzpicture}[baseline=0pt]
\draw (0pt, 0pt) circle (7pt);
\draw (7pt, 0pt)--(-7pt, 0pt);
\draw (3.5pt, 6.06218pt)--(3.5pt, -6.06218pt);
\draw (-3.5pt, 6.06218pt)--(-3.5pt, -6.06218pt);
\fill (6.06218pt, -3.5pt) circle (1pt);
\end{tikzpicture}
-
\begin{tikzpicture}[baseline=0pt]
\draw (0pt, 0pt) circle (7pt);
\draw (7pt, 0pt)--(-3.5pt, -6.06218pt);
\draw (3.5pt, 6.06218pt)--(-7pt, 0pt);
\draw (-3.5pt, 6.06218pt)--(3.5pt, -6.06218pt);
\fill (6.06218pt, -3.5pt) circle (1pt);
\end{tikzpicture}
-
\begin{tikzpicture}[baseline=0pt]
\draw (0pt, 0pt) circle (7pt);
\draw (7pt, 0pt)--(-7pt, 0pt);
\draw (3.5pt, 6.06218pt)--(-3.5pt, -6.06218pt);
\draw (-3.5pt, 6.06218pt)--(3.5pt, -6.06218pt);
\fill (6.06218pt, -3.5pt) circle (1pt);
\end{tikzpicture}
~, \\
I_{3, 4} &:= \begin{tikzpicture}[baseline=0pt]
\draw (0pt, 0pt) circle (7pt);
\draw (7pt, 0pt)--(0pt, 7pt);
\draw (-7pt, 0pt)--(0pt, -7pt);
\fill (4.94975pt, -4.94975pt) circle (1pt);
\draw[font=\tiny] (10.5pt, 0pt) node {\mbox{$-$}};
\draw[font=\tiny] (-10.5pt, 0pt) node {\mbox{$-$}};
\end{tikzpicture}
-
\begin{tikzpicture}[baseline=0pt]
\draw (0pt, 0pt) circle (7pt);
\draw (7pt, 0pt)--(0pt, 7pt);
\draw (-7pt, 0pt)--(0pt, -7pt);
\fill (4.94975pt, -4.94975pt) circle (1pt);
\draw[font=\tiny] (10.5pt, 0pt) node {\mbox{$+$}};
\draw[font=\tiny] (-10.5pt, 0pt) node {\mbox{$+$}};
\end{tikzpicture}
+
\begin{tikzpicture}[baseline=0pt]
\draw (0pt, 0pt) circle (7pt);
\draw (7pt, 0pt)--(0pt, -7pt);
\draw (0pt, 7pt)--(-7pt, 0pt);
\fill (4.94975pt, -4.94975pt) circle (1pt);
\draw[font=\tiny] (10.5pt, 0pt) node {\mbox{$-$}};
\draw[font=\tiny] (0pt, 10.5pt) node {\mbox{$-$}};
\end{tikzpicture}
-
\begin{tikzpicture}[baseline=0pt]
\draw (0pt, 0pt) circle (7pt);
\draw (7pt, 0pt)--(0pt, -7pt);
\draw (0pt, 7pt)--(-7pt, 0pt);
\fill (4.94975pt, -4.94975pt) circle (1pt);
\draw[font=\tiny] (10.5pt, 0pt) node {\mbox{$+$}};
\draw[font=\tiny] (0pt, 10.5pt) node {\mbox{$+$}};
\end{tikzpicture}
+
\begin{tikzpicture}[baseline=0pt]
\draw (0pt, 0pt) circle (7pt);
\draw (7pt, 0pt)--(-7pt, 0pt);
\draw (0pt, 7pt)--(0pt, -7pt);
\fill (4.94975pt, -4.94975pt) circle (1pt);
\draw[font=\tiny] (10.5pt, 0pt) node {\mbox{$-$}};
\draw[font=\tiny] (0pt, 10.5pt) node {\mbox{$-$}};
\end{tikzpicture}
-
\begin{tikzpicture}[baseline=0pt]
\draw (0pt, 0pt) circle (7pt);
\draw (7pt, 0pt)--(-7pt, 0pt);
\draw (0pt, 7pt)--(0pt, -7pt);
\fill (4.94975pt, -4.94975pt) circle (1pt);
\draw[font=\tiny] (10.5pt, 0pt) node {\mbox{$+$}};
\draw[font=\tiny] (0pt, 10.5pt) node {\mbox{$+$}};
\end{tikzpicture}
-
\begin{tikzpicture}[baseline=0pt]
\draw (0pt, 0pt) circle (7pt);
\draw (7pt, 0pt)--(3.5pt, 6.06218pt);
\draw (-3.5pt, 6.06218pt)--(-7pt, 0pt);
\draw (-3.5pt, -6.06218pt)--(3.5pt, -6.06218pt);
\fill (6.06218pt, -3.5pt) circle (1pt);
\end{tikzpicture}
-
\begin{tikzpicture}[baseline=0pt]
\draw (0pt, 0pt) circle (7pt);
\draw (7pt, 0pt)--(3.5pt, -6.06218pt);
\draw (3.5pt, 6.06218pt)--(-3.5pt, 6.06218pt);
\draw (-7pt, 0pt)--(-3.5pt, -6.06218pt);
\fill (6.06218pt, -3.5pt) circle (1pt);
\end{tikzpicture}
-
\begin{tikzpicture}[baseline=0pt]
\draw (0pt, 0pt) circle (7pt);
\draw (7pt, 0pt)--(3.5pt, 6.06218pt);
\draw (-3.5pt, 6.06218pt)--(-3.5pt, -6.06218pt);
\draw (-7pt, 0pt)--(3.5pt, -6.06218pt);
\fill (6.06218pt, -3.5pt) circle (1pt);
\end{tikzpicture}
\\
& \phantom{=}
-
\begin{tikzpicture}[baseline=0pt]
\draw (0pt, 0pt) circle (7pt);
\draw (7pt, 0pt)--(3.5pt, -6.06218pt);
\draw (3.5pt, 6.06218pt)--(-7pt, 0pt);
\draw (-3.5pt, 6.06218pt)--(-3.5pt, -6.06218pt);
\fill (6.06218pt, -3.5pt) circle (1pt);
\end{tikzpicture}
-
\begin{tikzpicture}[baseline=0pt]
\draw (0pt, 0pt) circle (7pt);
\draw (7pt, 0pt)--(-3.5pt, 6.06218pt);
\draw (3.5pt, 6.06218pt)--(-7pt, 0pt);
\draw (-3.5pt, -6.06218pt)--(3.5pt, -6.06218pt);
\fill (6.06218pt, -3.5pt) circle (1pt);
\end{tikzpicture}
-
\begin{tikzpicture}[baseline=0pt]
\draw (0pt, 0pt) circle (7pt);
\draw (7pt, 0pt)--(-3.5pt, 6.06218pt);
\draw (3.5pt, 6.06218pt)--(3.5pt, -6.06218pt);
\draw (-7pt, 0pt)--(-3.5pt, -6.06218pt);
\fill (6.06218pt, -3.5pt) circle (1pt);
\end{tikzpicture}
-
\begin{tikzpicture}[baseline=0pt]
\draw (0pt, 0pt) circle (7pt);
\draw (7pt, 0pt)--(-3.5pt, -6.06218pt);
\draw (3.5pt, 6.06218pt)--(3.5pt, -6.06218pt);
\draw (-3.5pt, 6.06218pt)--(-7pt, 0pt);
\fill (6.06218pt, -3.5pt) circle (1pt);
\end{tikzpicture}
-
\begin{tikzpicture}[baseline=0pt]
\draw (0pt, 0pt) circle (7pt);
\draw (7pt, 0pt)--(-3.5pt, -6.06218pt);
\draw (3.5pt, 6.06218pt)--(-3.5pt, 6.06218pt);
\draw (-7pt, 0pt)--(3.5pt, -6.06218pt);
\fill (6.06218pt, -3.5pt) circle (1pt);
\end{tikzpicture}
-
\begin{tikzpicture}[baseline=0pt]
\draw (0pt, 0pt) circle (7pt);
\draw (7pt, 0pt)--(3.5pt, 6.06218pt);
\draw (-3.5pt, 6.06218pt)--(3.5pt, -6.06218pt);
\draw (-7pt, 0pt)--(-3.5pt, -6.06218pt);
\fill (6.06218pt, -3.5pt) circle (1pt);
\end{tikzpicture}
-
\begin{tikzpicture}[baseline=0pt]
\draw (0pt, 0pt) circle (7pt);
\draw (7pt, 0pt)--(3.5pt, -6.06218pt);
\draw (3.5pt, 6.06218pt)--(-3.5pt, -6.06218pt);
\draw (-3.5pt, 6.06218pt)--(-7pt, 0pt);
\fill (6.06218pt, -3.5pt) circle (1pt);
\end{tikzpicture}
-
\begin{tikzpicture}[baseline=0pt]
\draw (0pt, 0pt) circle (7pt);
\draw (7pt, 0pt)--(-7pt, 0pt);
\draw (3.5pt, 6.06218pt)--(-3.5pt, 6.06218pt);
\draw (-3.5pt, -6.06218pt)--(3.5pt, -6.06218pt);
\fill (6.06218pt, -3.5pt) circle (1pt);
\end{tikzpicture}
-
\begin{tikzpicture}[baseline=0pt]
\draw (0pt, 0pt) circle (7pt);
\draw (7pt, 0pt)--(-3.5pt, 6.06218pt);
\draw (3.5pt, 6.06218pt)--(-3.5pt, -6.06218pt);
\draw (-7pt, 0pt)--(3.5pt, -6.06218pt);
\fill (6.06218pt, -3.5pt) circle (1pt);
\end{tikzpicture}
-
\begin{tikzpicture}[baseline=0pt]
\draw (0pt, 0pt) circle (7pt);
\draw (7pt, 0pt)--(-7pt, 0pt);
\draw (3.5pt, 6.06218pt)--(3.5pt, -6.06218pt);
\draw (-3.5pt, 6.06218pt)--(-3.5pt, -6.06218pt);
\fill (6.06218pt, -3.5pt) circle (1pt);
\end{tikzpicture}
-
\begin{tikzpicture}[baseline=0pt]
\draw (0pt, 0pt) circle (7pt);
\draw (7pt, 0pt)--(-3.5pt, -6.06218pt);
\draw (3.5pt, 6.06218pt)--(-7pt, 0pt);
\draw (-3.5pt, 6.06218pt)--(3.5pt, -6.06218pt);
\fill (6.06218pt, -3.5pt) circle (1pt);
\end{tikzpicture}
-
\begin{tikzpicture}[baseline=0pt]
\draw (0pt, 0pt) circle (7pt);
\draw (7pt, 0pt)--(-7pt, 0pt);
\draw (3.5pt, 6.06218pt)--(-3.5pt, -6.06218pt);
\draw (-3.5pt, 6.06218pt)--(3.5pt, -6.06218pt);
\fill (6.06218pt, -3.5pt) circle (1pt);
\end{tikzpicture}
~, \\
I_{3, 5} &:= \begin{tikzpicture}[baseline=0pt]
\draw (0pt, 0pt) circle (7pt);
\draw (7pt, 0pt)--(0pt, 7pt);
\draw (-7pt, 0pt)--(0pt, -7pt);
\fill (4.94975pt, -4.94975pt) circle (1pt);
\draw[font=\tiny] (10.5pt, 0pt) node {\mbox{$-$}};
\draw[font=\tiny] (-10.5pt, 0pt) node {\mbox{$-$}};
\end{tikzpicture}
-
\begin{tikzpicture}[baseline=0pt]
\draw (0pt, 0pt) circle (7pt);
\draw (7pt, 0pt)--(0pt, 7pt);
\draw (-7pt, 0pt)--(0pt, -7pt);
\fill (4.94975pt, -4.94975pt) circle (1pt);
\draw[font=\tiny] (10.5pt, 0pt) node {\mbox{$+$}};
\draw[font=\tiny] (-10.5pt, 0pt) node {\mbox{$+$}};
\end{tikzpicture}
-
\begin{tikzpicture}[baseline=0pt]
\draw (0pt, 0pt) circle (7pt);
\draw (7pt, 0pt)--(3.5pt, 6.06218pt);
\draw (-3.5pt, 6.06218pt)--(-7pt, 0pt);
\draw (-3.5pt, -6.06218pt)--(3.5pt, -6.06218pt);
\fill (6.06218pt, -3.5pt) circle (1pt);
\end{tikzpicture}
-
\begin{tikzpicture}[baseline=0pt]
\draw (0pt, 0pt) circle (7pt);
\draw (7pt, 0pt)--(3.5pt, 6.06218pt);
\draw (-3.5pt, 6.06218pt)--(-3.5pt, -6.06218pt);
\draw (-7pt, 0pt)--(3.5pt, -6.06218pt);
\fill (6.06218pt, -3.5pt) circle (1pt);
\end{tikzpicture}
-
\begin{tikzpicture}[baseline=0pt]
\draw (0pt, 0pt) circle (7pt);
\draw (7pt, 0pt)--(-3.5pt, 6.06218pt);
\draw (3.5pt, 6.06218pt)--(-7pt, 0pt);
\draw (-3.5pt, -6.06218pt)--(3.5pt, -6.06218pt);
\fill (6.06218pt, -3.5pt) circle (1pt);
\end{tikzpicture}
+
\begin{tikzpicture}[baseline=0pt]
\draw (0pt, 0pt) circle (7pt);
\draw (7pt, 0pt)--(-3.5pt, -6.06218pt);
\draw (3.5pt, 6.06218pt)--(3.5pt, -6.06218pt);
\draw (-3.5pt, 6.06218pt)--(-7pt, 0pt);
\fill (6.06218pt, -3.5pt) circle (1pt);
\end{tikzpicture}
-
\begin{tikzpicture}[baseline=0pt]
\draw (0pt, 0pt) circle (7pt);
\draw (7pt, 0pt)--(3.5pt, 6.06218pt);
\draw (-3.5pt, 6.06218pt)--(3.5pt, -6.06218pt);
\draw (-7pt, 0pt)--(-3.5pt, -6.06218pt);
\fill (6.06218pt, -3.5pt) circle (1pt);
\end{tikzpicture}
-
\begin{tikzpicture}[baseline=0pt]
\draw (0pt, 0pt) circle (7pt);
\draw (7pt, 0pt)--(-7pt, 0pt);
\draw (3.5pt, 6.06218pt)--(-3.5pt, 6.06218pt);
\draw (-3.5pt, -6.06218pt)--(3.5pt, -6.06218pt);
\fill (6.06218pt, -3.5pt) circle (1pt);
\end{tikzpicture}
+
\begin{tikzpicture}[baseline=0pt]
\draw (0pt, 0pt) circle (7pt);
\draw (7pt, 0pt)--(-3.5pt, 6.06218pt);
\draw (3.5pt, 6.06218pt)--(-3.5pt, -6.06218pt);
\draw (-7pt, 0pt)--(3.5pt, -6.06218pt);
\fill (6.06218pt, -3.5pt) circle (1pt);
\end{tikzpicture}
+
\begin{tikzpicture}[baseline=0pt]
\draw (0pt, 0pt) circle (7pt);
\draw (7pt, 0pt)--(-7pt, 0pt);
\draw (3.5pt, 6.06218pt)--(-3.5pt, -6.06218pt);
\draw (-3.5pt, 6.06218pt)--(3.5pt, -6.06218pt);
\fill (6.06218pt, -3.5pt) circle (1pt);
\end{tikzpicture}
~, \\
I_{2, 1} &:= \begin{tikzpicture}[baseline=0pt]
\draw (0pt, 0pt) circle (7pt);
\draw (7pt, 0pt)--(0pt, 7pt);
\draw (-7pt, 0pt)--(0pt, -7pt);
\fill (4.94975pt, -4.94975pt) circle (1pt);
\end{tikzpicture}
-
\begin{tikzpicture}[baseline=0pt]
\draw (0pt, 0pt) circle (7pt);
\draw (7pt, 0pt)--(-7pt, 0pt);
\draw (0pt, 7pt)--(0pt, -7pt);
\fill (4.94975pt, -4.94975pt) circle (1pt);
\end{tikzpicture}
~.  
\end{align*}

Finally, we would like to mention Proposition~\ref{classical}.  
\begin{proposition}\label{classical}
$I_{2, 1}$ is the Arnold invariant of long curve.  
\end{proposition}  
\begin{proof}
For Arnold $J^+$ invariant of long curves, the known formula \cite[Proposition~4.6 (4.6)]{Ito2010} plus \cite[Lemma~4.5 (4.4)]{Ito2010} implies that $I_{2, 1}$ equals $\frac{J^+ + r^2}{2}$ up to signs.   
\end{proof}
\section*{Acknowledgements}
The authors would like to thank Professor Takashi Inaba for his comments.   NI would like to thank Sara Yamaguchi for giving me her electronic data of figures of this paper.      
The work of NI was partially supported by MEXT KAKENHI Grant Number 20K03604.

\bibliographystyle{plain}
\bibliography{Ref}
\end{document}